\documentclass[11pt, reqno]{amsart}
\usepackage{amsmath,amssymb,amsfonts,amscd,hyperref,color}
\usepackage[utf8]{inputenc}
\usepackage{csquotes}

\usepackage[abs]{overpic}
\usepackage{verbatim}

\newcommand{\Hmm}[1]{\leavevmode{\marginpar{\tiny%
$\hbox to 0mm{\hspace*{-0.5mm}$\leftarrow$\hss}%
\vcenter{\vrule depth 0.1mm height 0.1mm width \the\marginparwidth}%
\hbox to
0mm{\hss$\rightarrow$\hspace*{-0.5mm}}$\\\relax\raggedright #1}}}

\newtheorem{thm}{Theorem}

\newtheorem{pro}[thm]{Proposition}

\theoremstyle{definition}

\newtheorem*{rem}{Remark}

\numberwithin{equation}{section}
\newcommand{\Z}{{\mathbb Z}}
\newcommand{\R}{{\mathbb R}}

\newcommand{\N}{{\mathbb N}}

\newcommand{\D}{{\mathbb D}}

\newcommand{\al}{{\alpha}}
\newcommand{\be}{{\beta}}

\newcommand{\ph}{{\varphi}}
\newcommand{\eps}{{\varepsilon}}

\newcommand{\si}{{\sigma}}
\newcommand{\ka}{{\kappa}}

\renewcommand{\d}{\,\mathrm{d}}
\newcommand{\ep}{\varepsilon}

\newcommand{\longto}{\longrightarrow}

\newcommand{\s}{\sigma}

\renewcommand{\D}{\Delta}

\begin{document}
\title[Optimal Hardy Inequality for Fractional Laplacians]
{Optimal Hardy Inequality for Fractional Laplacians on the Integers}
\author[M.~Keller]{Matthias Keller}
\author[M.~Nietschmann]{Marius Nietschmann}

\address{M.~Keller and M.~Nietschmann, Institut f\"ur Mathematik, Universit\"at Potsdam
14476  Potsdam, Germany}

\email{matthias.keller@uni-potsdam.de}
\email{marius.nietschmann@uni-potsdam.de}
\date{\today}

\begin{abstract}
We study the fractional Hardy inequality on the integers. We prove the optimality of the Hardy weight  and hence affirmatively answer the question of sharpness of the constant. 
\end{abstract}
\maketitle

\section{Introduction}

In \cite{HardyPrecursor}, Hardy in 1919 proved an inequality which was shortly thereafter used to derive the \emph{classical Hardy inequality}
\[ \langle \D\ph, \ph \rangle \geq \langle w\ph, \ph \rangle \] 
with 
\[ w(x) = \frac1{4x^2} \] 
for $\ph\in C_c(\N_0)$ with $\ph(0)=0$, where $\D$ is the standard Laplacian on $\N_0$. We refer to \cite{HardyHistory} for more on the ``prehistory'' of Hardy's inequality. Indeed, the constant $1/4$ is known to be optimal in the sense that it cannot be replaced with a larger constant. Despite the optimality of the constant, \cite{KPPHardy} recently showed that the weight can be improved to a optimal one, i.e., to one that is largest possible in a certain sense. By definition, this includes the optimality of the constant.

Here, we study  a Hardy inequality 
\[ \langle \D^\s\ph, \ph \rangle \geq \langle w_\s\ph, \ph \rangle \] 
 of the fractional Laplacian $\D^\s$, $0<\s<1/2$, on $\Z$ for $\ph\in C_c(\Z)$ with 
\[ w_\s(x) 
= C_\s \frac1{|x|^{2\s}} + O\left(\frac1{|x|^{1+2\s}}\right) .\] 
The exact form of $w_\s$ which is given in terms of the Gamma function was obtained by  Ciaurri/Roncal in \cite{CR18}. However, the question concerning the optimality of the constant remained  open. Here, we show the optimality of their weight $ w_{\sigma} $ and as a consequence prove the optimality of the constant.

Hardy inequalities have a long mathematical tradition and are of great significance in various branches of mathematics such as functional analysis, partial differential equations, harmonic analysis, approximation theory and probability. To that end, a lot of work has been done to find optimal constants and critical weights in both the continuum and in the discrete setting (see e.g. \cite{Davies, DFP14, FT02, BGGP20, KPP20HardyOnManifoldsAndGraphs} and references therein). Typically, a Hardy inequality gives a lower bound of a quadratic form  by some weight which is then called a \emph{Hardy weight}.

The study of fractional Laplace operators can be traced back as far as to Riesz, see \cite{R39}. These operators appear as relativistic Schr\"odinger operators \cite{MV} in quantum mechanics and in the description of perturbed phenomena like turbulence \cite{B}, elasticity \cite{DPV}, laser design \cite{L}, and anomalous transport and diffusion \cite{M}. There are various ways to define the fractional Laplacian, for example, there is a spectral definition as a Fourier multiplier, a definition via the heat semigroup or a definition through harmonic extension. In fact, \cite{K17} discusses ten equivalent definitions for the fractional Laplacian on $\R^d$. Here, we rely on the definition involving a certain time-integral over the heat semigroup.

In the continuum, Hardy inequalities for the fractional Laplacian are rather well-understood. The explicit constants are known to be optimal, see \cite{FLS08} or \cite{H77, Y99} for earlier references. Similar results have been shown for the half-space in \cite{BD08}. In contrast, a fractional Hardy inequality in the discrete setting was obtained in \cite{CR18} but optimality of the constant remained an open question. Our contribution here is to show that their weight is optimal and thus affirmatively answer their question  about the optimality of the constant.

The structure of the paper is as follows. In the next section, we introduce the basic notions and formulate the main result. In Section~\ref{s:HardyWeights}, we present a family of Hardy weights  arising from superharmonic functions via the ground state transform. For a specific choice within this family we show in Section~\ref{s:optimality} optimality via a so called null-sequence argument.

\section{Set-Up and Main Result} \label{s:intro}

In this section, we introduce the main objects and concepts of this paper. We start by introducing the fractional Laplacian on the integers $ \Z $. Moreover, we recall the notion of optimality of Hardy weights in the sense of \cite{DFP14, KPP}. Finally, we state the main result.

The Laplacian $ \Delta $ on $ \ell^{2}(\Z) $ is the bounded operator acting as
\begin{align*}
	\Delta f(x)= (f(x)-f(x+1)) + (f(x)-f(x-1))
\end{align*}
for $ x\in \Z$.
The fractional Laplacian $ \Delta^{\sigma} $, $ \sigma\in (0,1) $, on $ \ell^{2}(\Z) $ is then given by the spectral calculus. By the boundedness of $ \Delta $, the fractional Laplacian is a bounded operator as well. Moreover, the spectral calculus yields by direct computation that $ \Delta^{\si} $ can be represented as
\begin{align*}
	\Delta^{\sigma}f(x)=\frac{1}{\Gamma(-\sigma)}\int_{0}^{\infty} (e^{-t\Delta}-1)f(x)\frac{\d t}{t^{1+\sigma}}
\end{align*}
where $ x\in \Z $, the operator $ e^{-t\Delta} $, $ t\ge0 $, is the semigroup of $ \Delta $ and $ \Gamma $ is the Gamma function. 
This operator is thoroughly studied in \cite{CRSTV}. We recall the most important facts here and refer to \cite{CRSTV} or \cite{Nietschmann22} for details.

Due to discreteness and homogeneity of $ \Z $, it can be directly seen that the operator $ \Delta^{\sigma} $, $ \si\in (0,1) $, can be written as a so called graph Laplacian or a kernel operator
\begin{align*}
	\Delta^{\sigma} f(x)=\sum_{y\in \Z}\ka_{\sigma}(x-y) (f(x)-f(y))
\end{align*}
where  the  kernel $ \ka_{\si} $ can be computed using the formula $ e^{-t\Delta}1=1 $ as $$  \ka_{\al}(x)=\frac{1}{|\Gamma(-\al)|}\int_{0}^{\infty}e^{-t\Delta}1_{0}(x)\frac{\d t}{t^{1+\al}}  $$
 for $ \al\in (-1/2,1) $, $   
 \al \neq 0 $, $ x\in \Z $, $ |x|>\al $,   as well as $ \kappa_{\al}(0)=0 $ for $ \al>0 $ and $ \kappa_{0}=1_{0} $. Here, $ 1_{0} $ denotes the characteristic function of $ 0 $.  We define $ \ka_{\al} $ here also for $ \al\in (-1/2,0] $ as it will be needed below.

Since the semigroup is  positivity improving (cf. \cite{KLW}), we infer that 
$ \kappa_{\al} $ is strictly positive. Hence, $ \Delta^{\sigma} $ is a graph Laplacian over $ \Z $ with respect to a non-locally finite weighted graph. It was shown in \cite[Lemma~9.2]{CRSTV} that
	\[ \kappa_\s (x)=\frac{4^\s \Gamma(\tfrac 1 2 + \s)}{\sqrt \pi |\Gamma(-\s)|} 
	\cdot \frac{\Gamma(|x|-\s)}{\Gamma(|x|+1+\s)}
= \frac{4^\s \Gamma(\tfrac 1 2 + \s)}{\sqrt \pi |\Gamma(-\s)|} 
\cdot \frac1{|x|^{1+2\s}} 
+ O\left(\frac1{|x|^{2+2\s}}\right) \]
for $ x\to\infty $. From the asymptotics  $ \ka_{\si}\in O(|\cdot|^{-1-2\si}) $, it is obvious that $ \Delta^{\si} $ can be extended to functions in the Banach space $$
B_{\si}=  \ell^1 (\Z, (1+|\cdot|)^{-1-2\si}) $$
for  $ \s\in (0,1) $. In particular, for $ \al>-1/2 $   and $ \si\in (0,1) $, we have
$  \ka_{ \al}\in B_{\si}  . $

Let us now turn to Hardy weights.
A non-trivial function $ w:\Z\longto [0,\infty) $ is called a \emph{Hardy weight} if
 for all compactly supported functions $ \ph\in C_{c}(\Z) $,
\begin{align*}
		\langle \Delta^{\si}\ph,\ph\rangle\ge 	\langle w\ph,\ph\rangle.
\end{align*}
By boundedness of $ \Delta^{\sigma} $, this inequality then extends directly to all functions in $ \ell^{2}(\Z) $. 
Next, we discuss optimality. A Hardy weight $ w $ is called \emph{critical} if any   Hardy weight $ w' \ge w$ satisfies $ w'=w $. If $ w $ is critical, then there  exists 
a unique (up to multiplicative constants) positive harmonic  function $ \psi $ (not necessarily in $ \ell^{2} $), i.e., $ (\Delta^{\sigma}-w)\psi =0$ which is called the \emph{Agmon ground state},  cf. \cite[Theorem 2.9]{KPP}. A critical Hardy weight $ w $ is called \emph{positive-critical} if $ \psi\in  \ell^{2}(\Z,w) $ and  \emph{null-critical} otherwise.
A Hardy weight $ w $ is called \emph{optimal} if the following two conditions are satisfied:
\begin{itemize}
	\item $ w $ is critical,
	\item $ w $ is null-critical.
\end{itemize}
It turns out  that if $ w $ is optimal, then it is also  \emph{optimal near infinity}, i.e., if $$ 	\langle \Delta^{\si}\ph,\ph\rangle\ge (1+\lambda)\langle w\ph,\ph\rangle   $$ for all $ \ph\in C_{c}(\Z\setminus K) $ for some compact $ K\subseteq \Z $, then $ \lambda \le0 $, cf. \cite{KPP,KovarikPinchover,Fischer}. This implies in particular, that the constant appearing within the Hardy weight is optimal.

Next, we present the Hardy weight which was obtained by \cite{CR18}.
Let 
\begin{align*}
	w_\s (x)
	=c_{\si} \cdot 
	\frac{\Gamma(|x|+\frac{1-2\s}4)
		\Gamma(|x|+\frac{3-2\s}4)}
	{\Gamma(|x|+\frac{3+2\s}4)
		\Gamma(|x|+\frac{1+2\s}4)} = 
	\frac{c_{\si}  }{|x|^{2\s}} 
	+ O\left(\frac1{|x|^{1+2\s}}\right)
\end{align*}\begin{align*}		
\mbox{with } c_{\si}= 4^\s \frac{\Gamma(\frac{1+2\s}4)^2}
	{\Gamma(\frac{1-2\s}4)^2}
\end{align*}
where $\s\in(0,1/2)$ which is strictly positive. We show  $ w_{\s} $ is an  optimal Hardy weight in the above sense. In particular, this affirmatively answers the open question in \cite{CR18} on the optimality of the constant.

\begin{thm}\label{main}
The function $ w_{\s} $ is an optimal Hardy weight of $ \Delta^{\sigma} $, $0<\s<\frac12$.	
\end{thm}

Let us  discuss the strategy of the argument of the proof. The proof in \cite{CR18} to show that $ w_{\sigma} $ is a Hardy weight uses the ground state transform for the Schr\"odinger operator $ \Delta^{\sigma}- w $ with $ w=(\Delta^{\sigma}u)/u $ where $ u>0 $ is a superharmonic function  of $ \Delta^{\sigma} $. To the best of our knowledge, this technique was first applied in the discrete setting in \cite{FSW08}. Indeed, \cite{CR18} show that  $ \ka_{-\al} $ are superharmonic functions of $ \Delta^{\si} $ for  $0<\s\leq\alpha<1/2$ and thus
\begin{align*}
	w_{\sigma,\alpha}=\frac{\Delta^{\sigma}\kappa_{-\alpha}}{\kappa_{-\alpha}}=\frac{\kappa_{\sigma-\alpha}}{\kappa_{-\alpha}}
\end{align*}
are Hardy weights. We  also provide a short and concise proof of the second equality in Proposition~\ref{p:ka} below.
By direct calculation one sees that $ w_{\si}=w_{\si,\al} $ for $$  \alpha=\frac{1+2\sigma}{4}  $$ and Theorem~\ref{main} above states the Hardy weight  is optimal in this case.

One may wonder how this particular choice of $ \alpha $ comes about. First of all, this choice of $ \al $ yields the same constant as in the continuous setting. In Theorem~\ref{thm} below, we show that $ w_{\sigma,\al}= {\kappa_{\sigma-\alpha}}/{\kappa_{-\alpha}} $ is critical for $ \al\leq ({1+2\sigma})/{4} $  and  $ \ka_{-\alpha}\not\in \ell^{2}(\mathbb{Z},w_{\si,\al}) $ for $ \al\ge ({1+2\sigma})/{4}. $ 
Thus, $ w_{\si,\al} $ is optimal exactly for $ \alpha=({1+2\sigma})/{4}  $ in which case it is equal to $ w_{\si} $ from above.

There is also another structural reason for the particular choice of $ \al $. In \cite{KPP18} based on \cite{DFP14}, it was shown that one obtains optimal Hardy weights via the super-solution construction 
\begin{align*}
	w=\frac{HG^{1/2}}{G^{1/2}}
\end{align*}
for the Green's function $ G $ of a Schr\"odinger operator $ H $ whenever $ G $ is proper and satisfies an anti-oscillation condition. For $H= \Delta^{\sigma} $, the Green's function can be seen to be given by $G= \kappa_{-\sigma} $. Furthermore, it turns out that
$ \kappa_{-\sigma}^{1/2} $ and $ \kappa_{-\alpha} $ with $  \alpha={(1+2\sigma)}/{4} $ share the same asymptotics. 
Thus, it seems plausible to consider $  \kappa_{-{(1+2\sigma)}/{4} } $  instead of $ \kappa_{-\sigma}^{1/2} $ in the super-solution construction to obtain an optimal Hardy weight.

Let us stress that the method developed in \cite{KPP18} to prove optimality for weighted graphs cannot be directly applied here. The reason is that the assumptions on the existence of a positive superharmonic function which is proper and satisfies an anti-oscillation condition directly implies local finiteness of the underlying graph. Yet, the weighted graph of the fractional Laplacian is non-locally finite as discussed above. Indeed, even the refined method developed by Hake \cite{Hake} does not apply here. Instead, we use the explicit asymptotics of the edge weights in our proof below.

\section{A Family of Hardy Weights}\label{s:HardyWeights}

In this section, we present a family of Hardy weights for the fractional Laplacian. 
It was  observed in \cite{CR18} that $ \ka_{-\al} $ are superharmonic functions of $ \Delta^{\si} $ for  $0<\s\leq\alpha<1/2$. For the sake of being self-contained, we give a short alternative argument of this fact.
\begin{pro}\label{p:ka}
For all  $0<\s\leq\alpha<1/2$, we have 
\begin{align*}
	\Delta^{\si}\kappa_{-\al}=\kappa_{\si-\al}.
\end{align*}
\end{pro}
\begin{proof}
	Define $\kappa_{\beta,\ep} : \Z \longto [0,\infty]$ for $\beta\in \R$ and $\ep>0$ as 
		\[\kappa_{\beta,\eps} 
		= \frac1{|\Gamma(-\beta)|} \int_0^\infty 
		e^{-t\ep} e^{-t\D}1_{0} \frac{\d t}{t^{1+\beta}} \] 
		for $ \beta\ne0 $ and $ \ka_{0,\ep}=1_{0} $ for $ \beta=0 $.
		
		By the spectral calculus, we have for $ \beta <0 $ 
		\begin{align*}
			 (\D+\ep)^{\beta} 1_{0}=\kappa_{\beta,\eps} 	
		\end{align*}
	and therefore, $ \ka_{\beta,\eps}\in \ell^{2}(\Z) $. 
	
	Moreover, for $ \beta>0 $ and $ f\in \ell^{2}(\Z) $, we obtain  by the spectral calculus
	\begin{align*}
		(\Delta+\eps)^{\beta}f(x)&=\frac{1}{\Gamma(-\be)}\int_{0}^{\infty}(e^{-t\eps}e^{-t\Delta}-1)f(x)\frac{\d t}{t^{1+\beta}}.\end{align*}
		We use $ e^{-t\Delta}1=1 $ and $ \int_{0}^{\infty}e^{-t\eps}t^{-1-\be}\d t=\eps^{\beta}\Gamma(-\beta) $ to arrive at
		\begin{align*}
			\ldots &=\frac{1}{\Gamma(-\be)}\int_{0}^{\infty}e^{-t\eps}(e^{-t\Delta}-1)f(x)\frac{\d t}{t^{1+\beta}} +\eps^{\be}f(x) \\
		&=\sum_{y\in \Z}\kappa_{\beta,\eps}(x-y)(f(x)-f(y))+\eps^{\beta} f(x).
	\end{align*}

Now, let $0<\s\leq\alpha<1/2$. So, by the spectral calculus we see that 
		\begin{align*}
			 (\D+\ep)^\s \kappa_{-\alpha,\eps} 
		= (\D+\ep)^\s (\D+\ep)^{-\alpha} 1_{0}
		= (\D+\ep)^{-(\alpha-\s)} 1_{0}
		=\kappa_{\s-\al,\eps}.
		\end{align*}  
Since 	$\ka_{\beta,\eps}(x)$ 	 converges monotonously to $ \kappa_{\beta} (x)$, $ x\in \Z $, as $ \eps\to0 $, the statement follows by the formulas  for $ \beta<0$ and $ \beta>0 $ above.
\end{proof}

From the proof, we can actually derive that $ \ka_{-\sigma} $ is  the Green's function of $ \Delta^{\sigma} $, $ 0<\sigma<1/2 $, i.e.,
\begin{align*}
\lim_{\eps\to 0}	(\Delta+\eps)^{-\sigma}1_{0}=\ka_{-\sigma}.
\end{align*}
In the  remark below, we shortly discuss how the formula above can be extended to the Banach spaces $  B_{-\al} $, $0< \al<1/2 $.
\begin{rem}
	One can define $ \Delta^{-\al} $ for $0< \al<1/2 $ on the Banach space $ B_{-\al}= \ell^1 (\Z, (1+|\cdot|)^{-1+2\al}) $ via
	\begin{align*}
		\Delta^{-\al}u(x)=\sum_{y\in \Z}\ka_{-\al}(x-y)u(y)=\frac{1}{\Gamma(\al)}\int_{0}^{\infty}e^{-t\Delta}u(x)\frac{\d t}{t^{1-\al}}.
	\end{align*}
Then, the proof above can be extended to functions in $ B_{-\al} $ for  $0<\s\leq\alpha<1/2$, i.e., for $ u\in B_{-\al} $, 
\begin{align*}
	\Delta^{\si}\Delta^{-\al}u= \Delta^{\si-\al}u.
\end{align*}
Here, the details can be found in \cite{Nietschmann22}.
\end{rem}

Next, we recall a basic fact which is a direct application of the ground state transform. First, observe that the quadratic form $ Q^{\sigma} $ associated to $ \Delta^{\sigma} $ acts as
\begin{align*}
Q^{\sigma}(f)=\frac{1}{2}\sum_{x,y\in\Z}\ka_{\sigma}(x-y) (f(x)-f(y))^{2}=	\langle \Delta^{\si}f,f\rangle
\end{align*}
on $ \ell^{2}(\Z) $ which can be seen by the virtue of Green's formula (cf. \cite{KLW}).
Furthermore, for  $0<\s\leq\alpha<1/2$ and $ \ph\in C_{c}(\Z) $, denote
\begin{align*}
Q^{\si}_{-\al}(\ph)	=\frac{1}{2}\sum_{x,y\in \Z}\ka_{\si}(x-y)\ka_{-\al}(x)\ka_{-\al}(y)(\ph(x)-\ph(y))^{2}.
\end{align*}

\begin{pro}[Ground state transform]\label{p:gst} Let  $0<\s\leq\alpha<1/2$ and $ w_{\sigma,\alpha}=\ka_{\si-\al} /\ka_{-\al}$. Then, for $ \ph\in C_{c}(\Z) $,
	\begin{align*}
		(Q^{\si}-w_{\s,\alpha})(\ph \ka_{-\al})=Q^{\si}_{-\al}(\ph).
	\end{align*}
In particular, $ w_{\sigma,\alpha} $ is a Hardy weight.
\end{pro}
\begin{proof} By Proposition~\ref{p:ka}, we have that $ \Delta^{\si}\ka_{-\al}=\ka_{\si-\al} $. Thus, $ (\Delta^{\sigma}-w_{\s,\al}) \ka_{-\al}=0 $ and 
	the statement follows by   the ground state transform, cf. \cite[Proposition~4.8]{KPP}.
\end{proof}

\section{Proof of Optimality}\label{s:optimality}

For the proof of optimality, we use the following criterion. 

\begin{pro}\label{p:crit} Let  $0<\s\leq\alpha<1/2$. Then, $ w_{\sigma,\alpha}=\ka_{\si-\al} /\ka_{-\al}$ is critical if and only if there is a sequence $ 0\leq e_{n}\le 1 $ in $ C_{c}(\Z) $ such that $ e_{n}\to 1 $ pointwise and
	\begin{align*}
		\sup_{n\in \N} Q_{-\al}^{\si}(e_{n})<\infty.
	\end{align*}
\end{pro}
\begin{proof} By \cite[Theorem~5.3]{KPP}, the form $ (Q^{\si}-w_{\si,\al}) $ is critical if and only if there is a null-sequence, i.e., there is a sequence $ 0\leq \eta_{n}\leq \ka_{-\al} $ in $ C_{c}(\Z) $ which converges pointwise to $ \ka_{-\al} $ and $  (Q^{\si}-w_{\si,\al})(\eta_{n})\to 0 $, $ n\to\infty $. By  the ground state transform above, we see that criticality of $ (Q^{\si}-w_{\si,\al}) $ is equivalent to recurrence of $ Q^{\si}_{-\al} $, i.e., existence of $0\leq e_{n}\leq 1  $ in $ C_{c}(\Z) $ such that $ e_{n}\to 1 $ pointwise and $ Q^{\si}_{-\al} (e_{n})\to 0$ as  $ n\to\infty $. By a Banach-Saks type argument, cf. \cite[Theorem~6.1, (i.d) $ \Leftrightarrow  $ (i.e)]{KLW}, the claim follows.
\end{proof}

Theorem~\ref{main} is a direct consequence of the following theorem which elaborates on the criticality on the Hardy weights $ w_{\sigma,\alpha} $ from above.
For the proof, we present a sequence $ (e_{n}) $ as in the proposition above.

\begin{thm}\label{thm}Let  $0<\s\leq\alpha<1/2$ and $ w_{\sigma,\alpha}=\ka_{\si-\al} /\ka_{-\al}$.
	\begin{itemize}
		\item [(a)] The Hardy weight $  w_{\sigma,\alpha} $ is critical for $\al \leq (1+2\s)/4 $.
		\item [(b)] The function $ \ka_{-\al}\in \ell^{2}(\Z,w_{\sigma,\alpha}) $ if and only if $\al < (1+2\s)/4 $. 
	In particular, the Hardy weight $  w_{\sigma,\alpha} $ is positive-critical for $ \al < (1+2\s)/4 $  and null-critical for $ \al = (1+2\s)/4$.
	\end{itemize}
\end{thm}
\begin{proof}
By Proposition~\ref{p:gst} the function $ w_{\si,\al} $ is a Hardy weight. 
To show criticality of $ w_{\si,\al} $, we present a sequence $ (e_{n}) $ with the properties as in the assumptions of Proposition~\ref{p:crit}.	
For $n\in\N$, we define $e_n : \Z \longto \R$, $e_n(0)=1$ and
\[ e_n(x) = 0\vee\left(1-\sqrt{\frac{\log|x|}{\log n}}\right),
\quad x\neq0. \] 
Clearly, the support of $e_n$ is $\{x\in\Z\mid|x|<n\}$, so $(e_n)$ is a sequence in $C_c(\Z)$, and $e_n\to1$ pointwise. So, by the proposition above, criticality of
$ w_{\si,\al}$ follows if we can show the uniform boundedness of $Q^\s_{-\al}(e_n)$. 
It is straight forward to  check using the asymptotics of $ \ka_{\beta}\in O(|\cdot|^{-1-2\beta}) $ that 
\begin{align*}
	Q^{\si}_{-\al}(e_{n})\leq C \sum_{1\le y<x} 
	\frac{(e_n(x)-e_n(y))^2}
	{(x-y)^{1+2\s}(xy)^{1-2\al}}.
\end{align*}
We divide the sum above into two sums over the sets $ \{ 1\le y<x\le n \} $ and $ \{1\le y< n <x \} $ since $ (e_{n}(x)-e_{n}(y))^2 $ vanishes outside of these two sets.

For the first sum, we use the subadditivity of the square root to estimate $  (e_{n}(x)-e_{n}(y))^{2}\leq \log (x/y)/\log n$ and obtain
\begin{align*}
	\sum_{1\le y<x\le n}  \frac{(e_n(x)-e_n(y))^2}
	{(x-y)^{1+2\s}(xy)^{1-2\al}}
	&\leq \frac1{\log n} \sum_{1\le y<x\le n}
	\frac{\log( x/y)}{(x-y)^{1+2\s}(xy)^{1-2\al}} 
      \\
    &\leq \frac1{\log n} \sum_{1\le y<n} \frac1{y^{3+2\s-4\al}} 
	\int_y^\infty \!\! \frac{\log(x/y) \d x}
	{(\frac xy-1)^{1+2\s}(\frac xy)^{1-2\al}} 
	  \\
	&= \frac1{\log n} \sum_{1\le y<n} \frac1{y^{2+2\s-4\al}} 
	\int_1^\infty \!\! \frac{\log t \d t}
	{(t-1)^{1+2\s}t^{1-2\al}}. 
\end{align*}
While we always have $\log t \leq t-1$, it is also true 
that for every $\beta>0$, there exists $t_\beta>0$ such that $\log t \leq (t-1)^{2\beta}$ for $t>t_\beta+1$. Hence, with $\beta = 1/2-\al$, 
\begin{align*}
\ldots 
	&\leq \frac1{\log n} \sum_{1\le y<n} \frac1{y^{2+2\s-4\al}} 
	\left(\int_0^{t_\beta} t^{(1-2\s)-1} \d t 
	+ \int_{t_\beta}^\infty t^{(2( \beta+\al-\si)-1)-1} \d t\right) \\ 
&\leq \frac{C_{\s,\al}}{\log n} \sum_{1\le y<n} \frac1{y^{2+2\s-4\al}}
\end{align*}
and the right hand side stays bounded by the $ \log $-asymptotics of the harmonic series if $\alpha=(1+2\s)/4$ and even goes to zero if $\alpha<(1+2\s)/4$. 

For the second  sum over the set $ \{1\le y< n<x\} $, we estimate the difference $ (e_{n}(x)-e_{n}(y)) ^{2}\leq 1$ and obtain 
\begin{align*}
	\sum_{1\le y< n<x} &\frac{(e_n(x)-e_n(y))^2}{(x-y)^{1+2\s}(xy)^{1-2\al}} 
		\leq \sum_{1\le y< n<x} \frac1{(x-y)^{1+2\s}(xy)^{1-2\al}} \\ 
		&\qquad\le \sum_{1\le y<n} \frac1{y^{2+2\s-4\al}} 
			\int_{\frac ny}^\infty (t-1)^{(2(\al-\s)-1)-1} \d t \\ 
		&\qquad= C_{\s,\al} \sum_{y<n} \frac1{y^{2+2\s-4\al}} 
			\left(\frac ny-1\right)^{2(\al-\s)-1} \\ 
		&\qquad\leq C_{\s,\al} \left(\int_0^{\frac ne} 
			+ \int_{\frac ne}^n\right) \frac1{y^{2+2\s-4\al}} 
			\left(\frac ny-1\right)^{2(\al-\s)-1} \d y.
\end{align*}
	We continue  to split up the integral and estimate
	\begin{align*}
\ldots &= \frac{C_{\s,\al}}{n^{2\s-4\al+1}} \int_0^{\frac ne} \left(\frac yn\right)^{2\al-1} 
\left(1-\frac yn\right)^{2(\al-\si)-1} \frac{\d y}n \\ 
		&\qquad\qquad\qquad+ \frac{C_{\s,\al}}{n^{2\s-4\al+1}} \int_{\frac ne}^n \left(\frac ny\right)^{2\s-4\al+1} \left(\frac ny-1\right)^{2(\al-\s)-1} \frac{\d y}y \\ 
	&\leq \frac{C_{\s,\al}}{n^{2\s-4\al+1}} \left[\left(1-\frac1e\right)^{2(\al-\s)-1} 
	\int_0^{\frac ne} \left(\frac yn\right)^{2\al-1} 
	\frac{\d y}n \right. \\ 
		&\qquad\qquad\qquad+ \left.\max\{e^{2\s-4\al+1},1\} 
		\int_{\frac ne}^n 
	\left(\log\frac ny\right)^{2(\al-\s)-1} \frac{\d y}y \right] \\
	&\le \frac{C_{\s,\al}}{n^{2\s-4\al+1}}\left[\int_0^{\frac1e} t^{2\al-1} \d t 
			+ \int_0^1 s^{2(\al-\s)-1} \d s\right]  
\end{align*}
where we substituted $t=y/n$ and $s = \log(n/y)$ such that 
$(\mathrm{d}y)/y = -\d s$. Observe that the right hand side stays bounded for $\al=({1+2\s})/4$ and even goes to zero if $\al<({1+2\s})/4$.
	
Putting  the two estimates together, we obtain 
$ \sup_{n\in\N} Q^{\si}_{-\al} (e_n) < \infty $ which shows criticality of $ (Q^{\si}-w_{\si,\alpha}) $ for $\alpha\le({1+2\s})/4$ by  Proposition~\ref{p:crit}.

To show null-criticality/positive-criticality, observe that $\psi= \ka_{-\al} $ is harmonic for $ (\Delta^{\si}-w_{\si,\al}) $. Since $ w_{\si,\al} $ is critical, $ \psi $ is the unique positive harmonic function (up to multiplicative constants) \cite[Theorem~5.3]{KPP}. From the asymptotics $\psi= \ka_{-\al}\in O(|\cdot|^{-1+2\al}) $ and $ w_{\si,\al}\in  O(|\cdot|^{-2\si}) $, we infer that
\begin{align*}
	\sum_{x\in \Z}\psi^{2}w_{\si,\al} \asymp   	\sum_{x\ge 1} {x^{2(-1+2\alpha)-2\si }}=   	\sum_{x\ge 1} {x^{(4\al-1-2\s)-1 }} 
	\begin{cases} 
		=\infty, &\al=\frac{1+2\s}4 \\ 
		<\infty, &\al<\frac{1+2\s}4, 
	\end{cases} 
\end{align*}
where $ \asymp $ means that there are two-sided estimates with positive constants.
This shows the statement about  null-criticality/positive-criticality.
\end{proof}

\begin{rem} The leading asymptotics of the Hardy weights $ w_{\sigma,\alpha} $ in $x$ differ in their constant but share the same decay which is $|x|^{-2\s}$. Numerically, it can be checked that these constants take their global maximum in $ [\s,1/2) $ at $ \alpha=(1+2\s) /4$ which then gives an optimal Hardy weight. This  implies that the weights $ w_{\sigma,\alpha} $ cannot be critical for  $ \alpha>(1+2\s) /4$: If they were critical, then they would also be null-critical as $ \ka_{-\al} $ is not in $ \ell^{2}(\Z,w_{\si,\al}) $ for $ \alpha>(1+2\s) /4$. However, in this case they would also be optimal at infinity which is a contradiction to the constant being strictly smaller than the one for $  \alpha=(1+2\s) /4 $.
\end{rem}
Finally, we deduce our main theorem from Theorem~\ref{thm} above.

\begin{proof}[Proof of Theorem~\ref{main}] By the theorem above, we have that 
	$  w_{\si,\al}=\ka_{\si-\al}/\ka_{-\al} $ is critical and null-critical and hence optimal for $ \al=({1+2\s})/4  $. Hence, the statement follows since
	 $w_{\sigma}= w_{\sigma,\al} $ for $ \al=({1+2\s})/4$.
\end{proof}

\textbf{Acknowledgement.} MK acknowledges the financial support of the DFG. Furthermore, the authors thank Philipp Hake, Yehuda Pinchover, Felix Pogorzelski and Luz Roncal for valuable discussions on the subject.

\bibliographystyle{alpha}
\bibliography{mynewbib.bib}

\end{document}